\documentclass[12pt]{amsart}
\usepackage[left=1in,right=1in,top=1in,bottom=1in]{geometry}
\usepackage{cite}
\usepackage{cancel}
\usepackage[ocgcolorlinks,linktoc=all]{hyperref}
\hypersetup{citecolor=blue,linkcolor=red}
\newtheorem{theorem}{Theorem}

\newtheorem{lemma}[theorem]{Lemma}

\newtheorem{corollary}[theorem]{Corollary}
\theoremstyle{definition}
\newtheorem{definition}[theorem]{Definition}

\theoremstyle{remark}
\newtheorem{remark}[theorem]{Remark}

\numberwithin{equation}{section}

\begin{document}

\title[Centro-affine normal flows with pinched Mahler volume]
 {Convex bodies with pinched Mahler volume
under the centro-affine normal flows}
\author[M.N. Ivaki]{Mohammad N. Ivaki}
\address{Institut f\"{u}r Diskrete Mathematik und Geometrie, Technische Universit\"{a}t Wien,
Wiedner Hauptstr. 8--10, 1040 Wien, Austria}
\curraddr{}
\email{mohammad.ivaki@tuwien.ac.at}


\begin{abstract}
We study the asymptotic behavior of smooth, origin-symmetric, strictly convex bodies under the centro-affine normal flows. By means of a stability version of the Blaschke-Santal\'{o} inequality, we obtain regularity of the solutions provided that initial convex bodies have almost maximum Mahler volume. We prove that suitably rescaled solutions converge sequentially to the unit ball in the $\mathcal{C}^{\infty}$ topology modulo $SL(n+1)$.
\end{abstract}
\maketitle
\section{Introduction}
The setting of this paper is $(n+1)$-dimensional Euclidean space, $\mathbb{R}^{n+1}.$  A compact convex subset of $\mathbb{R}^{n+1}$ with non-empty interior is called a \emph{convex body}. Write $\mathcal{F}^{n+1}$ and $\mathcal{F}^{n+1}_e$, respectively, for the set of strictly convex bodies which are smoothly embedded in $\mathbb{R}^{n+1}$ and for the set of all origin-symmetric convex bodies in $\mathcal{F}^{n+1}.$ The unit sphere is denoted by $\mathbb{S}^{n}$.

Let $K\in\mathcal{F}^{n+1}$ and $\nu:\partial K\to \mathbb{S}^{n}$ be the Gauss map of $\partial K.$ That is, at each point $x\in\partial K$, $\nu(x)$ is the unit outwards normal at $x$. Assume that $\mathcal{M}$ is an $n$-dimensional closed surface, smoothly embedded into $\mathbb{R}^{n+1}$ with a diffeomorphism $X_K$ and $X_K(\mathcal{M})=\partial K$.
The support function of $K$ as a function on the unit sphere is defined by
$$s(z):= \langle X_K(\nu^{-1}(z)), z \rangle,$$
for each $z\in\mathbb{S}^n$.

The matrix of the radii of curvature of $\partial K$ is denoted by $\mathfrak{r}=[\mathfrak{r}_{ij}]_{1\leq i,j\leq n}$ and the entries of $\mathfrak{r}$ are considered as functions on the unit sphere. They can be expressed in terms of the support function and its covariant derivatives as $\mathfrak{r}_{ij}:=\bar{\nabla}_i\bar{\nabla}_j s+s\bar{g}_{ij},$ where $[\bar{g}_{ij}]_{1\leq i,j\leq n}$ is the standard metric on $\mathbb{S}^{n}$ and $\bar{\nabla}$ is the standard Levi-Civita connection of $\mathbb{S}^{n}.$ The Gauss curvature of $\partial K$ is denoted by $\mathcal{K}$, and as a function on $\partial K$, it is also related to the support function of the convex body by \[\frac{1}{\mathcal{K}\circ\nu^{-1}}:=S_n=\det_{\bar{g}}[\bar{\nabla}_i\bar{\nabla}_js+\bar{g}_{ij}s]:=\frac{\det [\mathfrak{r}_{ij}]}{\det{[\bar{g}_{ij}]}}.\]
In the sequel, for simplicity we usually denote $\mathcal{K}\circ\nu^{-1}$ by $\mathcal{K}.$
Finally, the eigenvalues of $[\mathfrak{r}_{ij}]$ with respect to the metric $[\bar{g}_{ij}]$ are denoted by $\lambda_1\leq \lambda_2\leq\cdots\leq \lambda_n$ for $1\leq i\leq n.$
Thus, $\lambda$ is an eigenvalues of $[\mathfrak{r}_{ij}]$ with respect to the metric $[\bar{g}_{ij}]$  if and only if $\det[\mathfrak{r}_{ij}-\lambda \bar{g}_{ij}]=0.$  The principal curvatures of $\partial K$ are $\{\kappa_i(x):=\frac{1}{\lambda_i(\nu(x))}\}$ for $1\leq i\leq n$ and $x\in\partial K.$

We now proceed to describe the flow that we will study in this paper. Assume $p> 1$ is a fixed real number and let $K_0\in \mathcal{F}^{n+1}_e$. A family of convex bodies $\{K_t\}_t\subset\mathcal{F}^{n+1}_e$ given by the smooth embeddings $X:\mathcal{M}\times[0,T)\to \mathbb{R}^{n+1}$ is said to be a solution of the $p$ centro-affine normal flow, in short $p$-flow, with the initial data $X_{K_0}$, if the following evolution equation is satisfied:
\begin{equation}\label{e: p flow ev of x}
\left\{
  \begin{array}{ll}
     \partial_{t}X(x,t)=-\left(\frac{\mathcal{K}(x,t)}{\langle X(x,t), \nu(x,t)\rangle^{n+2}}\right)^{\frac{p}{p+n+1}-\frac{1}{n+2}}\mathcal{K}^{\frac{1}{n+2}}(x,t)\, \nu(x,t), & \hbox{} \\
    X(\cdot,0)=X_{K_0}.
  \end{array}
\right.
\end{equation}
In this equation, $0<T<\infty$ is the maximal time that the solution exists, and $\nu(x,t)$ is the unit normal to the hypersurface $ X(\mathcal{M},t)=\partial K_t$ at $X(x,t).$
The short time existence and uniqueness of solutions for a smooth and strictly convex initial hypersurface follow from the strict parabolicity of the equation, and it was shown in \cite{S}. As the name centro-affine suggests, solutions of the $p$ centro-affine normal flow are $SL(n+1)$ invariant while Euclidean translations of an initial convex body will lead to different solutions, since translations affect the support function of the convex body which appears in the speed of the centro-affine normal flow. It is clear from the definition of the support function that as convex bodies $\{K_t\}$ evolve by (\ref{e: p flow ev of x}) their corresponding support functions solve the following fully nonlinear equation:
 \begin{equation}\label{e: p flow ev of s}
\partial_t s(\cdot,t)=-s\left(\frac{\mathcal{K}}{s^{n+2}}\right)^{\frac{p}{p+n+1}}(\cdot,t),~~ s(\cdot ,t)=s_{K_t}(\cdot).
\end{equation}
The $p$-flow, $p>1$, was defined by Stancu in \cite{S} for the purpose of finding new global centro-affine invariants of smooth convex bodies in which a certain class of existing invariants arose naturally. Only the short time existence to the flow was then needed. Moreover, several interesting affine isoperimetric type inequalities were obtained via short time existence of the flow, \cite{S}. One also may consult \cite{S} to see an equivalent definition of the $p$-flow in terms of $SL(n+1)$ invariant quantities, e.q., in terms of a power of the centro-affine curvature, $\displaystyle\mathcal{K}/s^{n+2}$, and the centro-affine normal vector.

The long time behavior of the $p$-flow in $\mathbb{R}^2$ was studied by the author in \cite{Ivaki1,Ivaki3}. It was proved there that the area preserving $p$-flow with $p\in(1,\infty]$ evolves any convex body in $ \mathcal{F}^{n+1}_e$ to the unit disk, modulo $SL(2).$ The $p$-flow for $p=1$ is the well-known affine normal flow, which has been investigated by Sapiro and Tannenbaum \cite{ST4} for convex planar curves, by Angenent, Sapiro and Tannenbaum \cite{AST} for non-convex curves, and by Andrews \cite{BA4,BA1} in all dimensions. Andrews comprehensively studied the affine normal flow of compact, convex hypersurfaces in any dimension and showed that the volume-preserving flow evolves any convex initial bounded open set exponentially fast in the $\mathcal{C}^{\infty}$ topology to an ellipsoid. Moreover, for $n\geq 2$ existence and regularity of non-compact strictly convex solutions and ancient solutions of the affine normal flow have been investigated in \cite{LT} by Loftin and Tsui. See \cite{chen,ivaki5} for classification of compact, convex ancient solutions of the affine normal flow in $\mathbb{R}^2.$

In \cite{IS}, the author jointly with Stancu studied the asymptotic behavior of (\ref{e: p flow ev of s}) for $1\leq p<\frac{n+1}{n-1}.$ A curial ingredient there was the evolution equation of polar bodies. The polar body of $K$ with respect to the origin of $\mathbb{R}^{n+1}$, $K^{\ast}$, is the convex body defined as
\[
K^{\ast} = \{ y \in \mathbb{R}^{n+1} \mid x \cdot y \leq 1,\ \forall x
\in K \}.\]
It was proved in \cite{S} that if $\{K_t\}_{[0,T)}$ evolves by the $p$-flow, then $\{K_t^{\ast}\}_{[0,T)}$ is a solution of the
following evolution equation, the expanding $p$-flow (alternatively called the dual $p$-flow):
$$\partial_t s^{\ast}=s^{\ast}\left(\frac{\mathcal{K}^{\ast}}
{s^{\ast n+2}}\right)^{-\frac{p}{n+1+p}}.$$
This observation was the key to obtaining the regularity estimates in \cite{IS}. For a given convex body $K$, the volume of $K$, denoted by $V(K)$, is its Lebesgue measure as a subset of $\mathbb{R}^{n+1}$. In \cite{IS}, the following theorem was proved.
\begin{theorem}\cite{IS}
Let $1\leq p<\frac{n+1}{n-1}$ and $X_{K_0}$ be a smooth, strictly convex embedding of the boundary of $K_0\in\mathcal{F}^{n+1}_e.$ Then there exists a unique smooth solution $X:\mathcal{M}\times [0,T)\to\mathbb{R}^{n+1}$ of equation (\ref{e: p flow ev of x}) with initial data $X_{K_0}$. The rescaled hypersurfaces given by $\left(\frac{V(B)}{V(K_t)}\right)^{\frac{1}{n+1}}X(\mathcal{M},t)$ converge sequentially in the $\mathcal{C}^{\infty}$ topology to the unit sphere modulo $SL(n+1).$
\end{theorem}
The Mahler volume of an origin-symmetric convex body $K$ is defined as $V(K)V(K^{\ast}),$ which is an invariant quantity under the group $GL(n+1)$. The Blaschke-Santal\'{o} inequality states that the Mahler volume is maximized only for ellipsoids centered at the origin. That is, $V(K)V(K^{\ast})\leq \omega_{n+1}^2=V(B)^2$
with equality only for the origin centered ellipsoids, \cite{Bl}.
If a convex body $K$ satisfies $V(K)V(K^{\ast})> \frac{\omega_{n+1}^2}{1+\varepsilon}$, we say its Mahler volume, $V(K)V(K^{\ast})$, is $\varepsilon$-\emph{pinched}.
\begin{theorem}\label{thm: mainthm}
Let $p\geq\frac{n+1}{n-1}$ and $X_{K_0}$ be a smooth, strictly convex embedding of the boundary of $K_0\in\mathcal{F}^{n+1}_e$. Then there exists a unique smooth solution $X:\mathcal{M}\times [0,T)\to\mathbb{R}^{n+1}$ of equation (\ref{e: p flow ev of x}) with initial data $X_{K_0}$. Moreover, there exists an $\varepsilon>0$ such that if $K_0$ satisfies $V(K_0)V(K_0^{\ast})> \frac{\omega_{n+1}^2}{1+\varepsilon}$, then the rescaled hypersurfaces given by $\frac{1}{\left(\frac{2p(n+1)}{p+n+1}(T-t)\right)^{\frac{p+n+1}{2p(n+1)}}}X(\mathcal{M},t)$ converge sequentially in the $\mathcal{C}^{\infty}$ topology to the unit sphere modulo $SL(n+1).$
\end{theorem}
The restriction to origin-symmetric domains is natural in the centro-affine context. The technical difficulty in studying the $p$ centro-affine normal flows when $p>\frac{n+1}{n-1}$ is caused by the fact that for this range of $p$, $\mathcal{K}^{p},$ has homogeneity degree greater than one. The asymptotic behavior of convex hypersurfaces under geometric flows by speeds that are homogeneous functions of the principal curvatures of degree $\alpha>1$ has been a central focus of many papers. These papers are mainly divided into two categories, depending whether a pinching ratio on principal curvatures is assumed or not. In the former direction, one would like to show that solutions become spherical as they contract to points if a suitable pinching conditions on the principal curvatures of initial hypersurface is imposed. A pioneering work is \cite{Ch1}, where Chow treated flows by powers of the Gauss curvature. Other examples of such results are: powers of the mean curvature by Schulze \cite{Sch}, powers of the scalar curvature by Alessandroni and Sinestrari \cite{AS}, powers of the $m$-th mean curvature by Cabezas-Rivas and Sinestrari \cite{CRS} and a generalization of their result by Wu, Tian and Li \cite{WTL}, and flows by general functions of the principal curvatures by Andrews and McCoy \cite{BM}. In the latter category, the goal is to show without a pinching ratio on principal curvatures solutions become spherical as they contract to points. The first such results were obtained by Andrews \cite{BA2} in connection to Firey's conjecture. Other examples of such results are powers of the Gauss curvature by Andrews and Chen \cite{BX}, the squared norm of the second fundamental by Schn\"{u}rer \cite{Os}, and several more examples by Schulze and Schn\"{u}rer \cite{Sch}. These last mentioned results, in the second category, are all restricted to dimension three. Recently Guan and Ni using a new entropy functional, Chow's Harnack inequality and a beautiful trick, \emph{without} any pinching assumption, obtained the convergence of the normalized Gauss curvature flow in high dimensions \cite{GL}. In this regard, an alternative approach for obtaining a uniform lower bound on the Gauss curvature of the normalized solution is described in \cite{Ivaki4} where such a lower bound is obtained without Chow's Harnack inequality.

Continuity of the Mahler volume in the Hausdorff distance shows that a smooth convex body can have an arbitrarily large ratio $\lambda_n/\lambda_1$ while whose Mahler volume is close to the maximum value. This can be seen by cutting off negligible volumes from opposite caps of a ball and smoothing out the spherical edges. In Theorem \ref{thm: mainthm}, we do not impose any pinching condition on the principal curvatures of the initial convex body; we assume that the initial smooth, origin-symmetric convex body has $\varepsilon$-pinched Mahler volume for $\varepsilon>0$  small enough, to be determined later. Therefore, a weaker pinching condition is imposed compared to the conventional pinching condition. Additionally, we point out, as it will be shown in Corollary \ref{cor: pinching ratio}, that preservation of the pinching along the $p$-flow is an immediate corollary of the monotonicity of the Mahler volume under the $p$-flow.

The paper is structured as follows: Section 2 focuses on establishing basic properties of the $p$-flow. In Section 3, using a stability version of the Blaschke-Santal\'{o} inequality \cite{BB2}, and the monotonicity of the Mahler volume \cite{S} we obtain an estimate on the isoperimetric ratio, modulo $SL(n+1)$. In Section 4, we obtain a Harnack inequality for the $p$-flow which is the major result of the paper. We then proceed to obtain the regularity of solutions. Establishing a uniform upper bound on the speed of the flow is fairly easy. To obtain a uniform lower bound on the speed, we modify Andrews-McCoy's argument presented in \cite[Section 12]{BM}. In Section 5, we prove Theorem \ref{thm: mainthm}.
\section{Basic properties of the $p$-flow}
Given a convex body $K$, the inradius of $K$, $r_-(K)$, is the radius of the largest ball inscribed in $K$. The circumradius of $K$, $r_+(K)$, is the radius of the smallest ball containing $K$. For origin-symmetric convex bodies, the smallest and the largest balls as above are centered at the origin.
\begin{lemma} \label{lem: lower G1} Flow (\ref{e: p flow ev of s}) increases in time
$\min\limits_{z\in\mathbb{S}^{n}} \left(s\left(\frac{\mathcal{K}}{s^{n+2}}\right)^{\frac{p}{p+n+1}}\right)(z,t).$
\end{lemma}
\begin{proof} We compute the evolution equation of the speed. Let $\alpha:=1-\frac{(n+2)p}{p+n+1}$ and $\beta:=-\frac{p}{p+n+1}.$
\begin{align*}
\partial_t \left(s^{\alpha}S_n^{\beta}\right)=-\beta s^{\alpha}S_n^{\beta-1}(\dot{S}_n)^{ij}[\bar{\nabla}_i\bar{\nabla}_j\left(s^{\alpha}S_n^{\beta}\right)+\bar{g}_{ij}\left(s^{\alpha}S_n^{\beta}\right)]-\alpha s^{2\alpha-1}S_{2\beta},
\end{align*}
where $(\dot{S}_n)^{ij} := \frac{\partial S_n}{\partial  \mathfrak{r}_{ij}}$ is the derivative of $S_n$ with respect to the entry $\mathfrak{r}_{ij}$ of the radii of curvature matrix. Since $\beta$ and $\alpha$ are both non-positive, the claim follows from the maximum principle.
\end{proof}
\begin{lemma}\label{lem: upper G1}
For any smooth, strictly convex solution $\{K_t\}_{[0,t_1]}$ of (\ref{e: p flow ev of s}) with $0<R_{-}\leq r_{-}(K_t)\leq r_{+}(K_t)\leq R_{+} < +\infty$ we have
\[\mathcal{K}(z,t)\leq C(n,R_+,R_-),\]
where $C$ is a constant depending on $n,R_{-},R_{+}$ and $\max_{z\in\mathbb{S}^n}\mathcal{K}(z,0).$
\end{lemma}
\begin{proof}
We apply a standard technique due to Tso \cite{Tso}. For simplicity, we may set $\alpha:=1-\frac{(n+2)p}{p+n+1}$ and $\beta:=-\frac{p}{p+n+1}.$ Employing the maximum principle, we will prove that $\Psi(z,t)$ defined by
\[\Psi:=\frac{s^{\alpha}S_n^{\beta}}{s-R_{-}/2}\]
is a bounded function from above with a bound only depending on $n, p, R_{-}, R_{+},$ and $\max\limits_{\mathbb{S}^n}\Psi(z,0).$
At the point where the maximum of $\Psi$ is realized, we get
\[0=\bar{\nabla}_i\Psi=\bar{\nabla}_i \left(\frac{s^{\alpha}S_n^{\beta}}{s-R_{-}/2}\right) \hbox{~and~}  \bar{\nabla}_i\bar{\nabla}_j \Psi\leq 0.\]
Thus, we obtain
$\frac{\bar{\nabla}_i (s^\alpha S_n^\beta)}{s-R_-/2}=\frac{(s^\alpha S_n^\beta) \bar{\nabla}_i s}{(s-R_-/2)^2},$
and consequently
\begin{equation}\label{e: tso}
\bar{\nabla}_i\bar{\nabla}_j\left(s^{\alpha}S_n^{\beta}\right)+\bar{g}_{ij}\left(s^{\alpha}S_n^{\beta}\right)\leq
\frac{s^{\alpha}S_n^{\beta}\mathfrak{r}_{ij}-(R_{-}/2)s^{\alpha}S_n^{\beta}\bar{g}_{ij}}{s-R_{-}/2}.
\end{equation}
To apply the parabolic maximum principle, we calculate the time derivative of $\Psi:$
\begin{align*}
\partial_t\Psi=&-\frac{\beta s^{\alpha}S_n^{\beta-1}}{s-R_{-}/2}(\dot{S}_n)^{ij}
\left[\bar{\nabla}_i\bar{\nabla}_j\left(s^{\alpha}S_n^{\beta}\right)+\bar{g}_{ij}\left(s^{\alpha}S_n^{\beta}\right)\right]\\
&+\frac{S_n^{\beta}}{s-R_{-}/2}\partial_t s^{\alpha}+\frac{s^{2\alpha}S_n^{2\beta}}{(s-R_{-}/2)^2}.
\end{align*}
Also notice that
\begin{equation}\label{e: one to last}
\frac{S_n^{\beta}}{s-R_{-}/2}\partial_t s^{\alpha}=-\alpha\Psi^2+\frac{\alpha R_{-}}{2}\frac{s^{2\alpha-1}S_n^{2\beta}}{(s-R_-/2)^2}\leq -\alpha\Psi^2.
\end{equation}
Hence, using inequalities (\ref{e: tso}) and (\ref{e: one to last}) we infer that at the point where the maximum of $\Psi$ is reached we have
\begin{equation}\label{e: last step tso}
\partial_t\Psi\leq\Psi^2\left(-n\beta-\alpha+1+\frac{\beta R_{-}}{2}\mathcal{H}\right),
\end{equation}
where the symbol $\mathcal{H}=\sum\limits_i \kappa_i$ stands for the mean curvature. We consider two cases. First, we may assume that the maximum of $\Psi$ is achieved at a time $t>0$. In this case, we have $\partial_t\Psi\geq0.$
So inequality (\ref{e: last step tso}) implies that
\[\mathcal{K}\leq \left(\frac{\mathcal{H}}{n}\right)^{n}\leq \left(\frac{4(n+1)}{nR_{-}}\right)^n\]
at the point where the maximum of $\Psi$ is reached. This in turn implies that $\mathcal{K}(z,t)\leq C_1(n,R_+,R_-).$
Second, the maximum of $\Psi$ may occur at $t=0$, we then have
\[\frac{s\left(\frac{\mathcal{K}}{s^{n+2}}\right)^{\frac{p}{p+n+1}}}{s-R_-/2}(z,t)\leq \frac{s\left(\frac{\mathcal{K}}{s^{n+2}}\right)^{\frac{p}{p+n+1}}}{s-R_-/2}(z,0)\leq C(R_+,R_-)\mathcal{K}(z,0).\]
Thus $\mathcal{K}(z,t)\leq C_2(R_+,R_-)\mathcal{K}(z,0).$ Taking $C=C_1+C_2$ completes the proof.
\end{proof}
\begin{remark}\label{rem: rem}
Notice that if the Gauss curvature is bounded from above, then a lower bound on the principal curvatures implies an upper bound on the principal curvatures.
\end{remark}
\begin{lemma}[Lower bound on the principal curvatures]\cite{IS}\label{lem: lower P}
Let $\{K_t\}_{[0,t_1]}$ be a smooth strictly convex solution  of (\ref{e: p flow ev of s}) with $0<R_{-}\leq r_{-}(K_t)\leq r_{+}(K_t)\leq R_{+}< +\infty$ and suppose $C_1\leq S_n\leq C_2$ for all $t\in[0,t_1].$
Then there exist constants $C$ and $C'$ depending on $n,p,R_{-}$, $R_{+},C_1$ and $C_2$ such that
\[\kappa_i(\cdot,t)\geq \frac{1}{C+C't^{-(n-1)}},~\forall t\in[0,t_1].\]
\end{lemma}
\begin{theorem}\label{theorem: zero volume}
The solution of (\ref{e: p flow ev of x}) exists on a maximal finite time interval $[0,T)$ and $\lim\limits_{t\to T}V(K_t)=0.$
\end{theorem}
\begin{proof}
Let $B_0$ be a large ball that encloses $K_0$. It can be easily verified that the solution to the $p$-flow starting at $B_0$, denoted by $B_t$, shrinks to the origin in finite time. By the containment principle, $ K_t\subseteq B_t$, therefore $T$ must be finite. Now suppose that, contrary to our claim, $V(K_t)$ does not tend to zero. Thus, we must have $s(\cdot,t)\geq R_-$ on $[0,T)$, for some $R_->0$. By Lemmas \ref{lem: lower G1}, \ref{lem: upper G1}, \ref{lem: lower P} and Remark \ref{rem: rem} the principal curvatures remain uniformly bounded on $[0,T)$ from below and above. Consequently, evolution equation (\ref{e: p flow ev of s}) is uniformly parabolic on $[0,T)$, and bounds on higher derivatives of the support function follow \cite{Krylov-Safonov,K,Krylov}, see also \cite{Tsai}. Hence, we can extend the solution smoothly past $T$, contradicting the maximality of $T$.
\end{proof}
\section{Bounding the isoperimetric ratio}
We will state a stability version of the Blaschke-Santal\'{o} inequality which has been proved by K. Ball and K.J. B\"{o}r\"{o}czky in \cite{BB2} for $n\geq 2$ (See \cite{Ivaki2} for the planar case.). We only present their result in the class of origin-symmetric convex bodies. To do so, we start with the definition of the Banach-Mazur distance.
\begin{definition} The Banach-Mazur distance of two origin-symmetric convex bodies $K$ and $L$ is defined by
\[\delta_{BM}(K,L)=\ln\min\{\lambda\geq 1: L\subseteq AK\subseteq\lambda L~\mbox{~for~}~A\in GL(n+1)\}.\]
\end{definition}
\begin{theorem}[Stability of the Blaschke-Santal\'{o} inequality]\cite{BB2} Let $n\geq 2$ and $K\in\mathcal{F}^{n+1}_e$ satisfies
$V(K)V(K^{\ast})>\frac{\omega_{n+1}^2}{1+\varepsilon}$
for an $\varepsilon>0,$ then for some $\gamma$ depending only on $n$, we have
$\delta_{BM}(K,B)\leq \gamma \varepsilon^{\frac{2}{3(n+2)}}|\log \varepsilon|^{\frac{4}{3(n+2)}}.$
\end{theorem}
The following result is proved by Stancu in \cite{S}.
\begin{theorem}[Monotonicity of the Mahler volume]\cite{S}\label{theorem: monoton}
Let $\{K_t\}$ be a smooth, strictly convex solution of (\ref{e: p flow ev of s}). Then $V(K_t)V(K_t^{\ast})$ is non-decreasing along the $p$-flow. The monotonicity is strict unless $K_t$ is an ellipsoid centered at the origin.
\end{theorem}
Combining these last two theorems we obtain the next corollary.
\begin{corollary}
Let $\{K_t\}_{[0,T)}$ be a smooth, strictly convex solution of (\ref{e: p flow ev of s}). If
$V(K_0)V(K_0^{\ast})>\frac{\omega_{n+1}^2}{1+\varepsilon}$
for an $\varepsilon>0,$ then
$\delta_{BM}(K_t,B)\leq \gamma \varepsilon^{\frac{2}{3(n+2)}}|\log \varepsilon|^{\frac{4}{3(n+2)}}.$
\end{corollary}
Now from the definition of the Banach-Mazur distance we have:
\begin{corollary}\label{cor: pinching ratio}
Let $\{K_t\}_{[0,T)}$ be a smooth, strictly convex solution of (\ref{e: p flow ev of s}). If
$V(K_0)V(K_0^{\ast})>\frac{\omega_{n+1}^2}{1+\varepsilon}$
for an $\varepsilon>0,$ then for each time $t$ there exists a special linear transformation $A_{t}\in SL(n+1),$ such that
\[\frac{r_{+}(A_tK_t)}{r_{-}(A_{t}K_t)}\leq \delta:=\exp\left(\gamma \varepsilon^{\frac{2}{3(n+2)}}|\log \varepsilon|^{\frac{4}{3(n+2)}}\right).\]
\end{corollary}
Set $\alpha:=-1+\frac{2(n+1)p}{p+n+1}.$ In the remainder of this paper, we take an $\varepsilon>0$ small enough such that
\begin{align}\label{e: pinching}
&\left(1-\left[\delta^{1+\alpha}-0.5\right]^{\frac{1}{1+\alpha}}\right)\nonumber\\
&=\left(1-\left[\exp\left((1+\alpha)\gamma \varepsilon^{\frac{2}{3(n+2)}}|\log \varepsilon|^{\frac{4}{3(n+2)}}\right)-0.5\right]^{\frac{1}{1+\alpha}}\right)>0.
\end{align}
Notice that this assumption in particular implies that $1\leq\delta<1.5^{\frac{1}{1+\alpha}}.$
We now restate Theorem \ref{thm: mainthm}.
\begin{theorem}
Let $p\geq\frac{n+1}{n-1}$ and $X_{K_0}$ be a smooth, strictly convex embedding of $K_0\in\mathcal{F}^{n+1}_e.$ Then there exists a unique smooth solution $X:\mathcal{M}\times [0,T)\to\mathbb{R}^{n+1}$ of (\ref{e: p flow ev of x}) with initial data $X_{K_0}$. Moreover, if for an $\varepsilon>0$ satisfying assumption (\ref{e: pinching}) the Mahler volume of $K_0$ is $\varepsilon$-pinched, then the family of rescaled hypersurfaces given by $\frac{1}{\left(\frac{2p(n+1)}{p+n+1}(T-t)\right)^{\frac{p+n+1}{2p(n+1)}}}X(\mathcal{M},t)$ converges sequentially in the $\mathcal{C}^{\infty}$ topology to the unit sphere modulo $SL(n+1).$
\end{theorem}
\section{Upper and lower bounds on the centro-affine curvature}
In this section, we obtain uniform upper and lower bounds on the centro-affine curvature. We begin by recalling an upper bound on the Gauss curvature established in \cite{IS}.
\begin{lemma}[Upper bound on the Gauss curvature]\cite{IS}\label{lem: upper G}
For any smooth, strictly convex solution $\{K_t\}_{[0,t_1]}$ of (\ref{e: p flow ev of s}) with $0<R_{-}\leq r_{-}(K_t)\leq r_{+}(K_t)\leq R_{+} < +\infty$, we have
$\mathcal{K}^{\frac{p}{n+p+1}}\leq \left(C+C't^{-\frac{np}{(n+1)(p+1)}}\right),$
where $C$ and $C'$ are constants depending on $n,p,R_{-}$ and $R_{+}.$
\end{lemma}
To obtain a lower bound on the centro-affine curvature, we may first establish a Harnack estimate. Although, Harnack inequality could be avoided, we present it here for future applications, such as stability of some inequalities \cite{Ivakig,Ivaki3}. In dimension two, the Harnack estimate for the $p$-flow was proved in \cite{Ivaki} with an application to classification of compact ancient solutions. To prove Lemma \ref{lem: Harnack est}, we closely follow Andrews \cite{BA5}.
\begin{lemma}[Harnack estimate]\label{lem: Harnack est} Let $\{K_t\}_t$ be a smooth solution of (\ref{e: p flow ev of s}). Then
\[\partial_t\left(s\left(\frac{\mathcal{K}}{s^{n+2}}\right)^{\frac{p}{p+n+1}}t^{\frac{np}{(p+1)(n+1)}}\right)\geq 0,\]
or equivalently
\[\partial_t\left(s\left(\frac{\mathcal{K}}{s^{n+2}}\right)^{\frac{p}{p+n+1}}\right)\geq -\frac{np}{(p+1)(n+1)t}\left(s\left(\frac{\mathcal{K}}{s^{n+2}}\right)^{\frac{p}{p+n+1}}\right).\]
\end{lemma}
\begin{proof}
For simplicity, we set $\gamma=-\frac{p}{p+n+1}.$ To prove the lemma, we will use the parabolic maximum principle to show that $\mathcal{R}$ defined by
\begin{equation}\label{def: expression R}
\mathcal{R}:=-t\partial_t\left(s^{1+(n+2)\gamma}S_n^{\gamma}\right)-\frac{\gamma}{\gamma-1/n}s^{1+(n+2)\gamma}S_n^{\gamma}
\end{equation}
is negative as long as the $p$-flow exists. Define $\mathcal{P}:=-\partial_t\left(s^{1+(n+2)\gamma}S_n^{\gamma}\right).$ Using the evolution equations of $s$ and $\mathfrak{r}_{ij}$ we get the following expression for $\mathcal{P}:$
\begin{align}\label{def: Q}
\mathcal{P}&=(1+(n+2)\gamma)s^{1+2(n+2)\gamma}S_{n}^{2\gamma}\nonumber\\
&+\gamma s^{1+(n+2)\gamma}S_{n}^{\gamma-1}(\dot{S}_n)^{ij}\left[\bar{\nabla}_i\bar{\nabla}_j\left(s^{1+(n+2)\gamma}S_n^{\gamma}\right)+\bar{g}_{ij}s^{1+(n+2)\gamma}S_n^{\gamma}\right]
\nonumber\\
&:=(1+(n+2)\gamma)s^{1+2(n+2)\gamma}S_{n}^{2\gamma}+\gamma s^{1+(n+2)\gamma}S_{n}^{\gamma-1}\mathcal{Q}.
\end{align}
To calculate the evolution equation of $\mathcal{P}$, we will repeatedly use  the evolution equations of $s$ and $\mathfrak{r}_{ij}.$
\begin{align*}
&\partial_t \mathcal{P}\\
&=-(1+(n+2)\gamma)(1+2(n+2)\gamma)s^{1+3(n+2)\gamma}S_{n}^{3\gamma}\\
&-2\gamma(1+(n+2)\gamma)s^{1+2(n+2)\gamma}S_n^{2\gamma-1}(\dot{S}_n)^{ij}
\left[\bar{\nabla}_i\bar{\nabla}_j\left(s^{1+(n+2)\gamma}S_n^{\gamma}\right)+\bar{g}_{ij}s^{1+(n+2)\gamma}S_n^{\gamma}\right]\\
&-\gamma(1+(n+2)\gamma)s^{1+2(n+2)\gamma}S_n^{2\gamma-1}(\dot{S}_n)^{ij}
\left[\bar{\nabla}_i\bar{\nabla}_j\left(s^{1+(n+2)\gamma}S_n^{\gamma}\right)+\bar{g}_{ij}s^{1+(n+2)\gamma}S_n^{\gamma}\right]\\
&-\gamma(\gamma-1)s^{1+(n+2)\gamma}S_{n}^{\gamma-2}
\left((\dot{S}_n)^{ij}\left[\bar{\nabla}_i\bar{\nabla}_j\left(s^{1+(n+2)\gamma}S_n^{\gamma}\right)+\bar{g}_{ij}s^{1+(n+2)\gamma}S_n^{\gamma}\right]\right)^2\\
&-\gamma s^{1+(n+2)\gamma}S_n^{\gamma-1}(\ddot{S}_n)^{ij,kl}\left[\bar{\nabla}_i\bar{\nabla}_j \mathcal{P}+\bar{g}_{ij}\mathcal{P}\right]\left[\bar{\nabla}_k\bar{\nabla}_l \mathcal{P}+\bar{g}_{kl}\mathcal{P}\right]\\
&-\gamma s^{1+(n+2)\gamma}S_n^{\gamma-1}
(\dot{S}_n)^{ij}\left[\bar{\nabla}_i\bar{\nabla}_j \mathcal{P}+\bar{g}_{ij}\mathcal{P}\right]\\
&\leq-(1+(n+2)\gamma)(1+2(n+2)\gamma)s^{1+3(n+2)\gamma}S_{n}^{3\gamma}\\
&-3\gamma(1+(n+2)\gamma)s^{1+2(n+2)\gamma}S_n^{2\gamma-1}\mathcal{Q}-\gamma(\gamma-1)s^{1+(n+2)\gamma}S_{n}^{\gamma-2}\mathcal{Q}^2\\
&-\frac{n-1}{n}\gamma s^{1+(n+2)\gamma}S_{n}^{\gamma-2}\mathcal{Q}^2-\gamma s^{1+(n+2)\gamma}S_n^{\gamma-1}
(\dot{S}_n)^{ij}\left[\bar{\nabla}_i\bar{\nabla}_j \mathcal{P}+\bar{g}_{ij}\mathcal{P}\right]\\
&=-(1+(n+2)\gamma)(1+2(n+2)\gamma)s^{1+3(n+2)\gamma}S_{n}^{3\gamma}\\
&-3\gamma(1+(n+2)\gamma)s^{1+2(n+2)\gamma}S_n^{2\gamma-1}\mathcal{Q}-\gamma\left(\gamma-1+\frac{n-1}{n}\right)s^{1+(n+2)\gamma}S_{n}^{\gamma-2}\mathcal{Q}^2\\
&-\gamma s^{1+(n+2)\gamma}S_n^{\gamma-1}
(\dot{S}_n)^{ij}\left[\bar{\nabla}_i\bar{\nabla}_j \mathcal{P}+\bar{g}_{ij}\mathcal{P}\right],
\end{align*}
where we used concavity of $S_n^{1/n}$:
\[\left((\ddot{S}_n)^{ij,kl}-\frac{n-1}{nS_n}(\dot{S}_n)^{ij}(\dot{S}_n)^{kl}\right)a_{ij}a_{lk}\leq 0\]
for every symmetric matrix $[a_{ij}]_{1\leq i,j\leq n}.$
By the definition of $\mathcal{Q}$, (\ref{def: Q}), we get
\begin{align}\label{e: q1}
\mathcal{Q}=\frac{\mathcal{P}-(1+(n+2)\gamma)s^{1+2(n+2)\gamma}S_{n}^{2\gamma}}{\gamma s^{1+(n+2)\gamma}S_{n}^{\gamma-1}}
\end{align}
and
\begin{align}\label{e: q2}
\mathcal{Q}^2=&\left(\frac{\mathcal{P}-(1+(n+2)\gamma)s^{1+2(n+2)\gamma}S_{n}^{2\gamma}}{\gamma s^{1+(n+2)\gamma}S_{n}^{\gamma-1}}\right)^2\nonumber\\
=&\frac{\mathcal{P}^2}{\gamma^2s^{2+2(n+2)\gamma}S_n^{2\gamma-2}}-\frac{2(1+(n+2)\gamma)}{\gamma^2}\frac{\mathcal{P}S_n^2}{s}\nonumber\\
&+\frac{(1+(n+2)\gamma)^2}{\gamma^2}s^{2(n+2)\gamma}S_n^{2\gamma+2}.
\end{align}
We replace $\mathcal{Q}$ and $\mathcal{Q}^2$ in the evolution equation of $\mathcal{P}$ by their equivalent expressions given in (\ref{e: q1}) and (\ref{e: q2}). We find that
\begin{align}\label{e: ev P}
&\partial_t \mathcal{P}\nonumber\\
&\leq-(1+(n+2)\gamma)(1+2(n+2)\gamma)s^{1+3(n+2)\gamma}S_{n}^{3\gamma}\nonumber\\
&-3\gamma(1+(n+2)\gamma)s^{1+2(n+2)\gamma}S_n^{2\gamma-1}\left(\frac{\mathcal{P}-(1+(n+2)\gamma)s^{1+2(n+2)\gamma}S_{n}^{2\gamma}}{\gamma s^{1+(n+2)\gamma}S_{n}^{\gamma-1}}\right)\nonumber\\
&-\gamma\left(\gamma-1+\frac{n-1}{n}\right)s^{1+(n+2)\gamma}S_{n}^{\gamma-2}\left(\frac{\mathcal{P}^2}{\gamma^2s^{2+2(n+2)\gamma}S_n^{2\gamma-2}}\right)\nonumber\\
&+\gamma\left(\gamma-1+\frac{n-1}{n}\right)s^{1+(n+2)\gamma}S_{n}^{\gamma-2}\left(\frac{2(1+(n+2)\gamma)}{\gamma^2}\frac{\mathcal{P}S_n^2}{s}\right)\nonumber\\
&-\gamma\left(\gamma-1+\frac{n-1}{n}\right)s^{1+(n+2)\gamma}S_{n}^{\gamma-2}\left(\frac{(1+(n+2)\gamma)^2}{\gamma^2}s^{2(n+2)\gamma}S_n^{2\gamma+2}\right)\nonumber\\
&-\gamma s^{1+(n+2)\gamma}S_n^{\gamma-1}
(\dot{S}_n)^{ij}\left[\bar{\nabla}_i\bar{\nabla}_j \mathcal{P}+\bar{g}_{ij}\mathcal{P}\right]\nonumber\\
&=-\gamma s^{1+(n+2)\gamma}S_n^{\gamma-1}
(\dot{S}_n)^{ij}\left[\bar{\nabla}_i\bar{\nabla}_j \mathcal{P}+\bar{g}_{ij}\mathcal{P}\right]\\
&+\left[(1+(n+2)\gamma)(2+(n+2)\gamma)-\frac{(\gamma-1/n)(1+(n+2)\gamma)^2}{\gamma}\right]s^{1+3(n+2)\gamma}S_{n}^{3\gamma}\nonumber\\
&+\left[-3(1+(n+2)\gamma)+\frac{2(\gamma-1/n)(1+(n+2)\gamma)}{\gamma}\right]s^{(n+2)\gamma}S_{n}^{\gamma}\mathcal{P}\nonumber\\
&-\frac{\gamma-1/n}{\gamma}\frac{\mathcal{P}^2}{s^{1+(n+2)\gamma}S_{n}^{\gamma}}.\nonumber
\end{align}
We proceed to obtain the evolution equation of $\mathcal{R}=t\mathcal{P}-\frac{\gamma}{\gamma-1/n}s^{1+(n+2)\gamma}S_n^{\gamma}$. First, notice that
\begin{align}\label{e: laplacian R}
-&\gamma s^{1+(n+2)\gamma}S_n^{\gamma-1}(\dot{S}_n)^{ij}\bar{\nabla}_i\bar{\nabla}_j\mathcal{R}=-t\gamma s^{1+(n+2)\gamma}S_n^{\gamma-1}(\dot{S}_n)^{ij}\bar{\nabla}_i\bar{\nabla}_j\mathcal{P}\\
&+\frac{\gamma^2}{\gamma-1/n}s^{1+(n+2)\gamma}S_{n}^{\gamma-1}(\dot{S}_n)^{ij}\bar{\nabla}_i\bar{\nabla}_j(s^{1+(n+2)\gamma}S_{n}^{\gamma})\nonumber.
\end{align}
Second, by identity (\ref{def: Q}), the evolution equation of $\mathcal{P}$ given by equation (\ref{e: ev P}), and identity (\ref{e: laplacian R}), it is straightforward to calculate
\begin{align*}
&\partial_t \mathcal{R}\\
&\leq-t\gamma s^{1+(n+2)\gamma}S_n^{\gamma-1}
(\dot{S}_n)^{ij}\left[\cancel{\bar{\nabla}_i\bar{\nabla}_j} \mathcal{P}+\bar{g}_{ij}\mathcal{P}\right]\\
&+t\left[(1+(n+2)\gamma)(2+(n+2)\gamma)-\frac{(\gamma-1/n)(1+(n+2)\gamma)^2}{\gamma}\right]s^{1+3(n+2)\gamma}S_{n}^{3\gamma}\\
&+t\left[-3(1+(n+2)\gamma)+\frac{2(\gamma-1/n)(1+(n+2)\gamma)}{\gamma}\right]s^{(n+2)\gamma}S_{n}^{\gamma}\mathcal{P}\\
&-t\frac{\gamma-1/n}{\gamma}\frac{\mathcal{P}^2}{s^{1+(n+2)\gamma}S_{n}^{\gamma}}+\mathcal{P}+\frac{\gamma}{\gamma-1}\bcancel{\mathcal{P}}
-\gamma s^{1+(n+2)\gamma}S_n^{\gamma-1}(\dot{S}_n)^{ij}\bar{\nabla}_i\bar{\nabla}_j\mathcal{R}\\
&+t\gamma s^{1+(n+2)\gamma}S_n^{\gamma-1}(\dot{S}_n)^{ij}\cancel{\bar{\nabla}_i\bar{\nabla}_j}\mathcal{P}
-\frac{\gamma^2}{\gamma-1/n}s^{1+(n+2)\gamma}S_{n}^{\gamma-1}(\dot{S}_n)^{ij}\bcancel{\bar{\nabla}_i\bar{\nabla}_j}(s^{1+(n+2)\gamma}S_{n}^{\gamma})\\
&-\frac{\gamma^2}{\gamma-1/n}s^{1+(n+2)\gamma}S_{n}^{\gamma-1}\bcancel{(\dot{S}_n)^{ij}}(s^{1+(n+2)\gamma}S_{n}^{\gamma})\bar{g}_{ij}\\
&+\frac{\gamma^2}{\gamma-1/n}s^{1+(n+2)\gamma}S_{n}^{\gamma-1}(\dot{S}_n)^{ij}(s^{1+(n+2)\gamma}S_{n}^{\gamma})\bar{g}_{ij}\\
&-\frac{\gamma(1+(n+2)\gamma)}{\gamma-1/n}s^{1+(n+2)\gamma}\bcancel{S_{n}^{\gamma-1}}\left(s^{1+2(n+2)\gamma}S_n^{2\gamma}\right)\\
&+\frac{\gamma(1+(n+2)\gamma)}{\gamma-1/n}s^{1+(n+2)\gamma}S_{n}^{\gamma-1}\left(s^{1+2(n+2)\gamma}S_n^{2\gamma}\right).
\end{align*}
Consequently,
\begin{align*}
&\partial_t \mathcal{R}\\
&=-\gamma s^{1+(n+2)\gamma}S_n^{\gamma-1}(\dot{S}_n)^{ij}\bar{\nabla}_i\bar{\nabla}_j\mathcal{R}\\
&+t\left[(1+(n+2)\gamma)(2+(n+2)\gamma)-\frac{(\gamma-1/n)(1+(n+2)\gamma)^2}{\gamma}\right]s^{1+3(n+2)\gamma}S_{n}^{3\gamma}\\
&+t\left[-3(1+(n+2)\gamma)+\frac{2(\gamma-1/n)(1+(n+2)\gamma)}{\gamma}\right]s^{(n+2)\gamma}S_{n}^{\gamma}\mathcal{P}\\
&-t\frac{\gamma-1/n}{\gamma}\frac{\mathcal{P}^2}{s^{1+(n+2)\gamma}S_{n}^{\gamma}}+\mathcal{P}-t\gamma s^{1+(n+2)\gamma}S_n^{\gamma-1}(\dot{S}_n)^{ij}\bar{g}_{ij}\mathcal{P}\\
&+\frac{\gamma^2}{\gamma-1/n}s^{2+2(n+2)\gamma}S_{n}^{2\gamma-1}(\dot{S}_n)^{ij}\bar{g}_{ij}+\frac{\gamma(1+(n+2)\gamma)}{\gamma-1/n}s^{2+3(n+2)\gamma}S_{n}^{3\gamma-1}.
\end{align*}
To make this last computation useful, using the definition of $\mathcal{R}$ we will replace $t\mathcal{P}$
by $\mathcal{R}+\frac{\gamma}{\gamma-1/n}s^{1+(n+2)\gamma}S_n^{\gamma}.$ Thus, at the point where the maximum of $\mathcal{R}$ is achieved we get
\begin{align*}
&\partial_t \mathcal{R}\leq\\
&\left[-3(1+(n+2)\gamma)+\frac{2(\gamma-1/n)(1+(n+2)\gamma)}{\gamma}\right]s^{(n+2)\gamma}S_{n}^{\gamma}\mathcal{R}\\
&+\frac{\gamma}{\gamma-1/n}\left[-3(1+(n+2)\gamma)+\frac{2(\gamma-1/n)(1+(n+2)\gamma)}{\gamma}\right]s^{1+2(n+2)\gamma}S_{n}^{2\gamma}\\
&-\frac{\gamma-1/n}{\gamma}\frac{1}{s^{1+(n+2)\gamma}S_{n}^{\gamma}}\left(\mathcal{R}+\cancel{\frac{\gamma}{\gamma-1/n}}s^{1+(n+2)\gamma}S_n^{\gamma}\right)\mathcal{P}
+\cancel{~\mathcal{P}~}\\
&-\gamma s^{1+(n+2)\gamma}S_n^{\gamma-1}(\dot{S}_n)^{ij}\bar{g}_{ij}\left(\mathcal{R}+\bcancel{\frac{\gamma}{\gamma-1/n}}s^{1+(n+2)\gamma}S_n^{\gamma}\right)\\
&+\bcancel{\frac{\gamma^2}{\gamma-1/n}}s^{2+2(n+2)\gamma}S_{n}^{2\gamma-1}(\dot{S}_n)^{ij}\bar{g}_{ij}+\frac{\gamma(1+(n+2)\gamma)}{\gamma-1/n}s^{2+3(n+2)\gamma}
S_{n}^{3\gamma-1}\\
&+t\left[(1+(n+2)\gamma)(2+(n+2)\gamma)-\frac{(\gamma-1/n)(1+(n+2)\gamma)^2}{\gamma}\right]s^{1+3(n+2)\gamma}S_{n}^{3\gamma}.
\end{align*}
There are two groups of terms: those that are multiple of $\mathcal{R}$ and those that are not. The latter group includes the terms on the second, fifth, and sixth lines which are all negative as $-1<\gamma\leq -1/(n+2).$ In view of the parabolic maximum principle, the former group is favorable. Since at the time $t=0$, we have $\mathcal{R}<0$
manifestly, we conclude that it remains negative.
\end{proof}
We continue with the following observation on obtaining lower bounds on the speed which first appeared in Smoczyk \cite{Sm} in his study of flow of star-shaped hypersurfaces by the mean curvature, and has been used in quite a few papers since then \cite{BM,BMZ,Ivaki3}.
\begin{lemma}\label{lem: app} For any $z\in\mathbb{S}^n$ the quantity
\[\left(1-\frac{np}{(p+1)(n+1)}\right)\left(s(z,t)-s(z,t_0)\right)+(t-t_0)\left(s\left(\frac{\mathcal{K}}{s^{n+2}}\right)^{\frac{p}{p+n+1}}\right)(z,t)\]
is nonnegative for all $t_0\leq t< T.$
\end{lemma}
\begin{proof}
Denote the left-hand side of the claimed inequality by $Q(t)$. We will prove $\frac{d}{dt}Q(t)\geq 0.$ Calculating the time derivative of $Q(t)$ yields
\[\frac{d}{dt}Q(t)=\left(\frac{np}{(p+1)(n+1)}\right)s\left(\frac{\mathcal{K}}{s^{n+2}}\right)^{\frac{p}{p+n+1}}
+(t-t_0)\frac{\partial}{\partial t}\left(s\left(\frac{\mathcal{K}}{s^{n+2}}\right)^{\frac{p}{p+n+1}}\right).\]
Notice that by Lemma \ref{lem: Harnack est}, after a time shifting, we have
\[\partial_t\left(s\left(\frac{\mathcal{K}}{s^{n+2}}\right)^{\frac{p}{p+n+1}}\right)\geq -\frac{np}{(p+1)(n+1)(t-t_0)}\left(s\left(\frac{\mathcal{K}}{s^{n+2}}\right)^{\frac{p}{p+n+1}}\right)\]
for all $t>t_0.$
The proof is complete; at the time $t=t_0$ we have $Q(t_0)=0.$
\end{proof}
Having Lemma \ref{lem: app} in hand, we modify Andrews-McCoy's argument from \cite[Section 12]{BM} to obtain a lower bound on the centro-affine curvature under the $p$-flow.
\begin{remark}
To simplify the notation, we write $\left(\frac{\mathcal{K}}{s^{n+2}}\right)(z,L)$ for the centro-affine curvature of the convex body $L$ at $z\in\mathbb{S}^n.$
\end{remark}
We recall the following property of the centro-affine curvature.
\begin{remark}\label{re: affin rem} For every $A\in SL(n+1)$ and $K\in\mathcal{F}^{n+1}_e$, we have
$$\min_{z\in\mathbb{S}^n}\frac{\mathcal{K}}{s^{n+2}}(z,K)=\min_{z\in\mathbb{S}^n}\frac{\mathcal{K}}{s^{n+2}}(z,AK)~\&~\max_{z\in\mathbb{S}^n}\frac{\mathcal{K}}{s^{n+2}}(z,K)=\max_{z\in\mathbb{S}^n}\frac{\mathcal{K}}{s^{n+2}}(z,AK).$$
\end{remark}
In the remainder of the present text, we set $\alpha:=-1+\frac{2(n+1)p}{p+n+1}$; $\alpha$ is the homogeneity degree of the speed of the $p$-flow.
\begin{lemma}[Lower bound on the centro-affine curvature]\label{lem: lower bound on the speed}
Let $K_0$ be a convex body whose Mahler volume is $\varepsilon$-pinched and the assumption (\ref{e: pinching}) is satisfied. Let $\{K_t\}_{[0,T)}$  be the smooth, strictly convex solution of (\ref{e: p flow ev of s}). Then there exist a constant $C>0$ and a time $t_{\ast}<T$, such that for each $t\geq t_{\ast}$ we have
\[\left(\frac{\mathcal{K}}{s^{n+2}}\right)^{\frac{p}{p+n+1}}(z,t)\geq \frac{C}{T-t}.\]
\end{lemma}
\begin{proof}
By Corollary \ref{cor: pinching ratio} for each $\tau\geq 0$ there exists a special linear transformation $A_{\tau}$, such that $\frac{r_{+}(A_{\tau}K_\tau)}{r_{-}(A_{\tau}K_\tau)}\leq \delta.$ Fix a $\tau\geq 0.$ Since $A_{\tau}K_{\tau}$ is origin-symmetric, the center of the maximal ball encompassed by $A_{\tau}K_{\tau}$ and the center of the minimal ball enclosing $A_{\tau}K_{\tau}$ are both located at the origin. Let $B_{r(t)}$ and $B_{R(t)}$ be solutions to the $p$-flow, respectively starting at $B_{r_{-}(A_{\tau}K_{\tau})}$ and $B_{\delta r_{-}(A_{\tau}K_{\tau})}$. The radii $R(t)$ and $r(t)$ are given by
\begin{equation}\label{eq: exp R}
R(t)=\left[(\delta r_{-}(A_{\tau}K_{\tau})^{1+\alpha}-(1+\alpha)(t-\tau)\right]^{\frac{1}{1+\alpha}}
\end{equation}
and
\[r(t)=\left[(r_{-}(A_{\tau}K_{\tau})^{1+\alpha}-(1+\alpha)(t-\tau)\right]^{\frac{1}{1+\alpha}}.\]
Notice that by the containment principle $B_{r(t)} \subseteq A_{\tau}K_t\subseteq B_{R(t)}$ for all $\tau \leq t\leq \tau+\frac{r_{-}(A_{\tau}K_{\tau})^{1+\alpha}}{1+\alpha},$ so we must have $T\geq\tau+\frac{r_{-}(A_{\tau}K_{\tau})^{1+\alpha}}{1+\alpha}.$ Take $\tau^{\ast}:=\tau+\frac{r_{-}(A_{\tau}K_{\tau})^{1+\alpha}}{2(1+\alpha)}$ and an arbitrary $z\in\mathbb{S}^n$. Set $\eta:=\left(1-\frac{np}{(p+1)(n+1)}\right)^{-1}>0.$ By Lemma \ref{lem: app} and equation (\ref{eq: exp R}) we obtain
\begin{align*}
\eta\left[\delta^{1+\alpha}-0.5\right]^{\frac{1}{1+\alpha}}&r_{-}(A_{\tau}K_{\tau})\left(\frac{\mathcal{K}}{s^{n+2}}\right)^{\frac{p}{p+n+1}}(z,A_{\tau}K_{\tau^{\ast}})\\
&=\eta R(\tau^{\ast})\left(\frac{\mathcal{K}}{s^{n+2}}\right)^{\frac{p}{p+n+1}}(z,A_{\tau}K_{\tau^{\ast}})\\
&\geq\eta\left(s\left(\frac{\mathcal{K}}{s^{n+2}}\right)^{\frac{p}{p+n+1}}\right)(z,A_{\tau}K_{\tau^{\ast}})\\
&\geq \frac{s_{A_{\tau}K_{\tau}}(z,\tau)-s_{A_{\tau}K_{\tau^{\ast}}}(z,\tau^{\ast})}{\tau^{\ast}-\tau}\\
&\geq \frac{2(r_{-}(A_{\tau}K_{\tau})-R(\tau^{\ast}))}{r_{-}(A_{\tau}K_{\tau})^{1+\alpha}}\\
&=\frac{2\left(1-\left[\delta^{1+\alpha}-0.5\right]^{\frac{1}{1+\alpha}}\right)}{r_{-}(A_{\tau}K_{\tau})^{\alpha}}.
\end{align*}
Therefore, we have
\[\left(\frac{\mathcal{K}}{s^{n+2}}\right)^{\frac{p}{p+n+1}}(z,A_{\tau}K_{\tau^{\ast}})\geq \frac{2(1+\alpha)C}{r_{-}(A_{\tau}K_{\tau})^{1+\alpha}},\]
for some positive constant $C.$
Recall that
\[T\geq\tau+\frac{r_{-}(A_{\tau}K_{\tau})^{1+\alpha}}{1+\alpha}=\tau^{\ast}+\frac{r_{-}(A_{\tau}K_{\tau})^{1+\alpha}}{2(1+\alpha)}.\]
 Thus, by Remark \ref{re: affin rem} we conclude that $K_{\tau^{\ast}}$ satisfies
\[\left(\frac{\mathcal{K}}{s^{n+2}}\right)^{\frac{p}{p+n+1}}(z,\tau^{\ast})\geq\frac{2(1+\alpha)C}{r_{-}(A_{\tau}K_{\tau})^{1+\alpha}}
\geq\frac{C}{T-\tau^{\ast}}.\]

To finish the proof, it suffices to show each $t\geq t_{\ast}:=\frac{r_{-}(A_{0}K_{0})^{1+\alpha}}{2(1+\alpha)}$
can be expressed as $t=\tau+\frac{r_{-}(A_{\tau}K_{\tau})^{1+\alpha}}{2(1+\alpha)}$ for a $\tau\geq 0:$
Define the function
$f$ on the time interval $[t_{\ast},T)$
by
\[f(\tau)=\tau+\frac{r_{-}(A_{\tau}K_{\tau})^{1+\alpha}}{2(1+\alpha)}-t.\]
Recall from Theorem \ref{theorem: zero volume} that $\lim\limits_{t\to T}V(A_tK_t)=\lim\limits_{t\to T}V(K_t)=0$. Hence, we must have $\lim\limits_{t\to T}r_{-}(A_tK_t)=0 .$ This implies that $\lim\limits_{\tau\to T}f(\tau)=T-\tau>0.$ On the other hand we have $f(0)\leq 0$. As $f$ is continuous, we conclude that there exists a $\tau$ such that $f(\tau)=0.$ Our argument to verify the lemma is now complete.
\end{proof}
\begin{remark}
For each $t\in[0,T)$ by the containment principle we have
\begin{equation}\label{e: controling the radii}
\frac{r_{-}(A_tK_t)^{1+\alpha}}{1+\alpha}\leq T-t\leq \frac{r_{+}(A_tK_t)^{1+\alpha}}{1+\alpha}\leq \frac{(\delta r_{-}(A_tK_t))^{1+\alpha}}{1+\alpha}.
\end{equation}
\end{remark}
\begin{lemma}[Upper bound on the centro-affine curvature]\label{lem: upper bound on the speed}
Let $K_0$ be a convex body whose Mahler volume is $\varepsilon$-pinched and the assumption (\ref{e: pinching}) is satisfied. Let $\{K_t\}_{[0,T)}$  be the smooth, strictly convex solution of (\ref{e: p flow ev of s}). Then there exists a constant $C'>0,$ such that on the time interval $[T/2,T)$ we have
\[\left(\frac{\mathcal{K}}{s^{n+2}}\right)^{\frac{p}{p+n+1}}(z,t)\leq \frac{C'}{T-t}.\]
\end{lemma}
\begin{proof}
Fix $t^{\ast}\in [T/2,T).$ Therefore,
$\tilde{K}^{t^{\ast}}_t=\frac{1}{(T-t^{\ast})^{\frac{1}{1+\alpha}}}A_{2t^{\ast}-T}K_{t^{\ast}+(T-t^{\ast})t}$
is a solution of (\ref{e: p flow ev of s}) on the time interval $[-1,0].$ By inequalities (\ref{e: controling the radii}), at the time $t=-1$
\[r_{-}(\tilde{K}^{t^{\ast}}_{-1})=\frac{r_{-}(A_{2t^{\ast}-T}K_{2t^{\ast}-T})}{(T-t^{\ast})^{\frac{1}{1+\alpha}}}\geq \frac{(2(1+\alpha))^{\frac{1}{1+\alpha}}}{\delta}\] and
\[r_{+}(\tilde{K}^{t^{\ast}}_{-1})=\frac{r_{+}(A_{2t^{\ast}-T}K_{2t^{\ast}-T})}{(T-t^{\ast})^{\frac{1}{1+\alpha}}}\leq \delta(2(1+\alpha))^{\frac{1}{1+\alpha}}.\]
By our remark after assumption (\ref{e: pinching}) we know that $\delta<1.5^{\frac{1}{1+\alpha}}$. Thus, by the containment principle for any time $t\in[-1,0]$ we get
\[r_{-}(\tilde{K}^{t^{\ast}}_{t})\geq \left((1+\alpha)\left(\frac{2}{\delta^{1+\alpha}}-1\right)\right)^{\frac{1}{1+\alpha}}\geq
\left(\frac{1+\alpha}{3}\right)^{\frac{1}{1+\alpha}},\]
and
\[r_{+}(\tilde{K}^{t^{\ast}}_{t})\leq \delta(2(1+\alpha))^{\frac{1}{1+\alpha}}< (3(1+\alpha))^{\frac{1}{1+\alpha}} .\]
This in turn implies, using Lemma \ref{lem: upper G}, that the centro-affine curvature of the convex body $\frac{1}{(T-t^{\ast})^{\frac{1}{1+\alpha}}}A_{2t^{\ast}-T}K_{t^{\ast}}=\tilde{K}^{t^{\ast}}_{0}$ is bounded by a positive constant $C'.$ Thus, the centro-affine curvature of $A_{2t^{\ast}-T}K_{t^{\ast}}$ and equivalently the centro-affine curvature of $K_{t^{\ast}}$ fulfils
\[\left(\frac{\mathcal{K}}{s^{n+2}}\right)^{\frac{p}{p+n+1}}(z,t^{\ast})\leq \frac{C'}{T-t^{\ast}}.\]
Since $t^{\ast}\in [T/2,T)$ is arbitrary and $C'$ is independent of $t^{\ast}$, the proof is complete.
\end{proof}
\section{Proof of Theorem \ref{thm: mainthm}}
Fix $t^{\ast}\in[\max\{3T/4,\frac{T+t_{\ast}}{2}\},T).$ We know
$\tilde{K}^{t^{\ast}}_t=\frac{1}{(T-t^{\ast})^{\frac{1}{1+\alpha}}}A_{2t^{\ast}-T}K_{t^{\ast}+(T-t^{\ast})t}$
is a solution of (\ref{e: p flow ev of s}) on the time interval $[-1,0]$ with
\[r_{-}(\tilde{K}^{t^{\ast}}_{t})\geq
\left(\frac{1+\alpha}{3}\right)^{\frac{1}{1+\alpha}},\]
and
\[r_{+}(\tilde{K}^{t^{\ast}}_{t})< (3(1+\alpha))^{\frac{1}{1+\alpha}} .\]
Since $2t^{\ast}-T\geq \max\{T/2,t_{\ast}\}$, by Lemmas \ref{lem: lower bound on the speed} and \ref{lem: upper bound on the speed} we get
\[\frac{C}{2(T-t^{\ast})}\leq \left(\frac{\mathcal{K}}{s^{n+2}}\right)^{\frac{p}{p+n+1}}(z,2t^{\ast}-T)\leq \frac{C'}{2(T-t^{\ast})}.\]
Hence, as $\alpha+1=\frac{2(n+1)p}{p+n+1}$ we conclude that the centro-affine curvature of $\tilde{K}^{t^{\ast}}_{-1}$ also satisfies
\[\frac{C}{2}\leq \left(\frac{\mathcal{K}}{s^{n+2}}\right)^{\frac{p}{p+n+1}}(\cdot,\tilde{K}^{t^{\ast}}_{-1})\leq \frac{C'}{2}.\]
To prove the main theorem, we recall two basic observations contained in Lemmas \ref{lem: lower G1} and \ref{lem: upper G1}:
\begin{enumerate}
  \item The minimum of the speed, $\min\limits_{\mathbb{S}^{n}} \left(s\left(\frac{\mathcal{K}}{s^{n+2}}\right)^{\frac{p}{p+n+1}}\right)$, is non-decreasing in time.
  \item The speed remains bounded from above as long as the inradius has a lower bound. Furthermore, the upper bound on the speed depends only on the speed and the circumradius of the initial body, and the lower bound on the inradii of evolving convex bodies.
\end{enumerate}
Using observations (1) and (2) we conclude that each $\tilde{K}^{t^{\ast}}_t$ for $t\in[-1,0]$ fulfils
\[C_1\leq s\left(\frac{\mathcal{K}}{s^{n+2}}\right)^{\frac{p}{p+n+1}}(z,t)\leq C_2,\]
for constants $C_1$ and $C_2$ independent of $t^{\ast}.$ Indeed, these constants are independent of $t^{\ast}$ as they only depend only on $C$, $C'$, and $\alpha.$
Consequently, for $t\in[-1,0],$ each $\tilde{K}^{t^{\ast}}_t$ satisfies $C_3\leq S_n(z,t)\leq C_4$ for some constants $C_3$ and $C_4$ independent of $t^{\ast}.$ Now Lemma \ref{lem: lower P} implies that there is a constant $C_5$ independent of $t^{\ast}$ such that each $\tilde{K}^{t^{\ast}}_t$ for $t\in[-1/2,0]$ fulfils $\kappa_i\geq C_5.$ Since $S_n\geq C_3$ by Remark \ref{rem: rem}, we can find a constant $C_6$ independent of $t^{\ast}$ such that $C_5\leq \kappa_i\leq C_6$ for each convex body $\tilde{K}^{t^{\ast}}_t$ and $\forall t\in[-1/2,0].$ Therefore, by \cite{Krylov-Safonov,Krylov} there are uniform bounds on all higher derivatives of the curvature of $\tilde{K}^{t^{\ast}}_t$ for $t\in[-1/2,0].$ In particular, $\tilde{K}^{t^{\ast}}_0=(T-t^{\ast})^{-\frac{1}{1+\alpha}}A_{2t^{\ast}-T}K_{t^{\ast}}$
has uniform $\mathcal{C}^{k}$ bounds independent of $t^{\ast}.$ Consequently, we can find a sequence of times $\{t_k\}_{k\in\mathbb{N}}$ such that as $\{t_k\}_k$ tends to $T$, the family of convex bodies $\left\{(T-t_k)^{-\frac{1}{1+\alpha}}A_{2t_k-T}K_{t_k}\right\}_k$ approaches in the $\mathcal{C}^{\infty}$ topology to a convex body $\tilde{K}_{T}.$ We now proceed to show that the limiting shape is the unit ball. To this end, we will start with a few facts on convex bodies.

A celebrated affine invariant quantity associated with $K$ is its affine surface area. More recently it was realized that the affine surface area belongs to the whole family of equi-affine invariant notions of surface areas. The homogeneous such surface areas are called $p$-affine surface areas and were defined by Lutwak \cite{Lutwak2} for $p>1$ in the context of the Firey-Brunn-Minkowski theory. For $p>1$ the $p$-affine surface area of a smooth, strictly convex body $K$ with the origin in its interior can be expressed as
\[
\Omega_p (K) =\int_{\mathbb{S}^{n}}\frac{s}{\mathcal{K}}\left(\frac{\mathcal{K}}{s^{n+2}}\right)^{\frac{p}{n+1+p}}
  d\mu_{\mathbb{S}^{n}}.
\]
A central inequality at the core of the centro-affine geometry is the $p$-affine isoperimetric inequality due to Lutwak \cite{Lutwak2} for $p> 1$, which is a generalization of the classical affine isoperimetric inequality: If $K$ has its centroid or Santal\'{o} point at the origin, then
\[
\frac{\Omega_p^{n+p+1}(K)}{ V^{n-p+1} (K)}\leq (n+1)^{n+p+1}\omega_{n+1}^{2p}.
\]
Equality holds only for ellipsoids centered at the origin.

The following result is an immediate corollary of the inequality given in \cite[Proposition 4.2]{S}.
\begin{theorem}\cite{S}
Let $\{K_t\}_{[0,T)}$ be a smooth, strictly convex solution of equation (\ref{e: p flow ev of s}). Then the $p$-affine isoperimetric ratio, $\frac{\Omega_{p}^{n+1+p}(K_t)}{V^{n+1-p}(K_t)}$, is non-decreasing along the $p$-flow. The monotonicity is strict unless $K_t$ is an ellipsoid centered at the origin.
\end{theorem}
Consequently, monotonicity of the $p$-affine isoperimetric ratio and Theorem \ref{theorem: zero volume} with a similar argument as in \cite{Ivaki1}, implies that $\tilde{K}_T$ must be an ellipsoid. Therefore, we get
$\lim\limits_{t_k\to T}\frac{\Omega_{p}^{n+1+p}(\tilde{K}^{t_k}_0)}{V^{n+1-p}(\tilde{K}^{t_k}_0)}=(n+1)^{n+p+1}\omega_{n+1}^{2p},$
and again by monotonicity of the $p$-affine isoperimetric ratio
$\lim\limits_{t\to T}\frac{\Omega_{p}^{n+1+p}(\tilde{K}^{t}_0)}{V^{n+1-p}(\tilde{K}^{t}_0)}=(n+1)^{n+1+p}\omega_{n+1}^{2p}.$
By the equality case in the $p$-affine isoperimetric inequality, we infer that
\[\lim_{t\to T}\frac{1}{(T-t)^{\frac{1}{1+\alpha}}}A_{2t-T}K_{t}=B\]
sequentially in the $\mathcal{C}^{\infty}$ topology, modulo $GL(n+1)$. On the other hand, observe that by the containment principle
$\frac{r_{-}(AK_t)^{1+\alpha}}{1+\alpha}\leq T-t\leq \frac{r_{+}(AK_t)^{1+\alpha}}{1+\alpha}$
for all $A\in SL(n+1)$. In particular,
$\frac{r_{-}(A_{2t-T}K_t)^{1+\alpha}}{1+\alpha}\leq T-t\leq \frac{r_{+}(A_{2t-T}K_{t})^{1+\alpha}}{1+\alpha}.$
Therefore
\[r_{-}\left(\frac{1}{((1+\alpha)(T-t))^{\frac{1}{1+\alpha}}}A_{2t-T}K_{t}\right)\leq 1\leq r_{+}\left(\frac{1}{((1+\alpha)(T-t))^{\frac{1}{1+\alpha}}}A_{2t-T}K_{t}\right).\]
From these last inequalities, it follows, modulo $SL(n+1)$, that
\[\lim_{t\to T}\frac{1}{((1+\alpha)(T-t))^{\frac{1}{1+\alpha}}}A_{2t-T}K_{t}=B\]
sequentially in the $\mathcal{C}^{\infty}$ topology. The proof is complete.

\textbf{Acknowledgment:} I am indebted to the referees whose comments and suggestions have led to improvements of this article.

\end{document}